\documentclass[a4paper,12pt]{article}
\usepackage{array,amsfonts,amsmath,graphicx,mathrsfs,amssymb,amsthm,latexsym}
\usepackage[latin1]{inputenc}
\newtheorem{thm}{Theorem}[section]
\newtheorem{lemma}[thm]{Lemma}

\def \cD {{\cal D}}

\def \cF {{\cal F}}
\def \cG {{\cal G}}

\def \cP {{\cal P}}

\begin{document}
\title{Resolvable $G$-designs of order $v$ and index $\lambda$}

\author {Mario Gionfriddo
\thanks{Supported  by PRIN and I.N.D.A.M (G.N.S.A.G.A.), Italy}\\
\small  Dipartimento di Matematica e Informatica \\
\small  Universit\`a di Catania \\
\small Catania\\
\small Italia\\
{\small \tt gionfriddo@dmi.unict.it}\\
Giovanni Lo Faro
\thanks{Supported  by PRIN and I.N.D.A.M (G.N.S.A.G.A.), Italy}\\
\small  Dipartimento di Matematica e Informatica \\
\small  Universit\`a di Messina \\
\small  Messina\\
\small Italia\\
{\small \tt lofaro@unime.it}\\
Salvatore Milici
\thanks{Supported by MIUR and by C. N. R. (G. N. S. A. G. A.), Italy}\\
\small Dipartimento di Matematica e Informatica \\
\small Universit\`a di Catania \\
\small Catania\\
\small Italia\\
{\small \tt milici@dmi.unict.it}  \\
Antoinette Tripodi
\thanks{Supported  by PRIN and I.N.D.A.M (G.N.S.A.G.A.), Italy}\\
\small  Dipartimento di Matematica e Informatica \\
\small  Universit\`a di Messina \\
\small  Messina\\
\small Italia\\
{\small \tt atripodi@unime.it}}

\date{ }
\maketitle

\begin{abstract}
In this paper we consider the problem concerning the existence of a resolvable $G$-design of order $v$ and index $\lambda$.
We solve the problem for the cases in which $G$ is a connected subgraph of $K_4$.
\end{abstract}

\vspace{5 mm}

\section{Introduction and definitions}\label{introduzione}

Let $v$ and $\lambda$ be positive integers , $\lambda K_{v}$ be the complete multigraph of order $v$ and index $\lambda$
and $G$ be a subgraph of $K_v$. A {\em $G$-design\/} of order $v$ and index $\lambda$
(denoted by ($\lambda K_{v},G)$-design),  is a decomposition of the edge set of $\lambda K_{v}$  into
subgraphs (called {\em blocks}) isomorphic to $G$. A ($\lambda K_{v},G$)-design  is
said to be {\em resolvable\/} if it is possible to partition the blocks
into  classes $\cP_i$ (often referred to as {\em parallel classes\/})
such that every vertex of $\lambda K_{v}$ appears in exactly one block of each $\cP_i$.
By simple calculation, we can obtain the following result.

\begin{lemma}
\label{lemmaP} If there exists a resolvable $(\lambda K_{v},G)$-design,
 then  $v\equiv\!  0\!\! \pmod{\mid V(G)\mid}$, $\lambda v(v-1) \equiv 0 \pmod{2\mid E(G)\mid}$, and
 $\lambda (v-1)\mid V(G)\mid \equiv 0 \pmod{2\mid E(G)\mid}$.
\end{lemma}

The Kirkman schoolgirl problem has developed the following question in the theory of resolvable $G$-designs: ``For a fixed graph $G$ and a index $\lambda$, what are necessary and sufficient conditions for the existence
of a resolvable  $(\lambda K_{v},G)$-design?"
The Kirkman schoolgirl problem is  this question for $G=K_3$ and $\lambda=1$, posed by Kirkman ($\cite{K}$) in $1847$ and solved by
Hanani, Ray-Chaudhuri and Wilson ($\cite{HRW}$) in 1969. The question has been studied for: $G=K_4$ and $\lambda =1,3$ by
Hanani, Ray-Chaudhuri and Wilson ($\cite{HRW}$);  $G=K_3$ and $\lambda =2$ by Hanani ($\cite{HA}$); $G=P_3$ and every admissible $\lambda$ by Horton ($\cite{H}$); $G=P_k, k\geq4$ and every admissible $\lambda$
by Bermond, Heinrich and Yu ($\cite{BHY}$); $ G=K_4-e$ and $\lambda =1$ by Ge, Ling, Colbourn,  Stinson, Whang and  Zhu ($\cite{CSZ,GL,W}$).

The existence of a resolvable decomposition  when $G$ is a subgraph of $K_4$ was studied separately already long ago:
\begin{itemize}
\item
There exists a resolvable $(\lambda K_{v},K_{2})$-design
 if and only if \ $v \equiv 0\pmod {2}$ and $\lambda \geq1$.
 \item
There exists a resolvable $(\lambda K_{v},P_{3})$-design
 if and only if \ $v \equiv 0\pmod {3}$ and $\lambda (v-1) \equiv0 \pmod{4}$ \ (\cite{H}).
\item
There exists a resolvable $(\lambda K_{v},K_{3})$-design
if and only if $\lambda \equiv0 \pmod{2}$  for \ $v \equiv 0\pmod {3}$, $v\neq6$,  or
 $\lambda \geq1$ for  \ $v \equiv 3\pmod {6}$  (\cite{HA, HRW}).
 \item
There exists a resolvable $(\lambda K_{v},P_{4})$-design
 if and only if \ $v \equiv 0\pmod {4}$ and $4 \lambda (v-1) \equiv0 \pmod{6}$ \ (\cite{BHY}).
\item
There exists a resolvable $( K_{v},K_{4}-e)$-design
 if and only if $v \equiv 16\pmod{20}$ $v\equiv 116\pmod{120}$ (\cite{CSZ,GL,W1}).
 \item
There exists a resolvable $(\lambda K_{v},K_{4})$-design
 if and only if $\lambda \equiv 0 \pmod{3}$ for \ $v \equiv 0,8\pmod {12}$  or $\lambda \geq 1$
 for \ $v \equiv 4\pmod {12}$ (\cite{HA, HRW}).
\end{itemize}
In this paper we shall focus our attention on the  problem  of the existence of  resolvable $(\lambda K_{v},G)$-designs when $G= C_4, K_3+e, K_{1,3}, K_4-e$, solving the spectrum problem for any connected subgraph of  $K_4$. 

 \bigskip \bigskip

In what follows, we will denote:
\begin{itemize}
\item
by $[a_1,a_2,\ldots, a_k]$  the path $P_k$, $k\geq 3$, having vertex set $\{a_1,a_2,\ldots, a_k\}$
and edge set $\{\{a_1,a_2\}, \{a_2,a_3\}, \ldots, \{a_{k-1},a_k\}\}$,
\item
by $(a_1, a_2,a_3,a_4)$ the 4-cycle $C_4$ having vertex set $\{a_1,a_2,a_3,a_4\}$ and edge set $\{\{a_1,a_2\},  \{ a_2, a_3\}, \{a_{3},a_4\}, \{a_{4},a_1\} \}$,
\item
by $(a_1;a_2,a_3,a_4)$ the {\em $3$-star\/} $K_{1,3}$ having vertex set $\{a_1,a_2, a_3,a_4\}$ and edge set $\{\{a_1,a_2\}, \{a_1,a_3\}, \{a_1,a_4\}\}$,
\item
by $(a_1,a_2,a_3;a_4)$ the graph $K_{4}-e$ having vertex set $\{a_1,a_2,a_3,a_4\}$ and edge set $\{\{a_1,a_2\}, \{a_1,a_3\}, \{a_2,a_3\}, \{a_1,a_4\}, \{a_2,a_4\}\}$ and
\item
 by $(a_1,a_2,a_3-a_4)$ the {\em kite\/} $K_{3}+e$ having vertex set $\{a_1,a_2,a_3,a_4\}$ and edge set $\{\{a_1,a_2\}, \{a_1,a_3\}, \{a_2,a_3\}, \{a_3,a_4\}\}$.
\end{itemize}

\section{Necessary conditions}

In this section we will give necessary conditions for the existence
of a resolvable $(\lambda K_{v},G)$-design. 

\begin{lemma}
\label{lemmaP1} If there exists a resolvable $(\lambda K_{v},G)$-design, with $G\in \{C_{4}, K_{3}+e\}$,
 then \ $v \equiv 0\pmod {4}$ and $\lambda \equiv 0 \pmod{2}$.
\end{lemma}

\begin{proof}
By Lemma  \ref{lemmaP} 
$$v \equiv 0\!\!\!\!\pmod {4},\ \ \lambda v(v-1) \equiv 0 \!\!\!\!\pmod{8},  \ \  \lambda (v-1) \equiv 0 \!\!\!\!\pmod{2},$$
and so the conclusion follows.
\end{proof}

\begin{lemma}
\label{lemmaP2} If there exists a resolvable $(\lambda K_{v},K_{1,3})$-design,
 then $v \equiv 0\pmod {4}$ and, in particular: 
 \begin{itemize}
 \item[$i)$] if \ $v \equiv 4\pmod {12}$, then $\lambda \equiv 0 \pmod{2}$;
 \item[$i)$] if  \ $v \equiv 0,8\pmod {12}$, then $\lambda \equiv 0 \pmod{6}$.
 \end{itemize}
\end{lemma}

\begin{proof}
By Lemma  \ref{lemmaP}  
$$v \equiv 0\!\!\!\!\pmod {4},\ \ \lambda v(v-1) \equiv 0 \!\!\!\!\pmod{6}, \ \ 2\lambda (v-1) \equiv 0 \!\!\!\!\pmod{3}.$$
Now proceeding as in the proof of Lemma 2.1 of \cite{KLMT} the number of the parallel classes must be $\equiv0 \pmod{4}$ and the conclusion follows.
\end{proof}

\begin{lemma}
\label{lemmaP3} If there exists a resolvable $(\lambda K_{v},K_{4}-e)$-design, 
 then $v \equiv 0\pmod {4}$ and, in particular: 
 \begin{itemize}
 \item[$i)$] if \ $v \equiv0, 4,8,12 \pmod {20}$, then $\lambda \equiv 0 \pmod{5}$;
 \item[$i)$] if  \ $v \equiv 16\pmod {20}$, then $\lambda$ is any potisive integer.
 \end{itemize}
\end{lemma}

\begin{proof}
By Lemma  \ref{lemmaP}  
$$v \equiv 0\!\!\!\!\pmod {4},\ \ \lambda v(v-1) \equiv 0 \!\!\!\!\pmod{10}, \ \ 2\lambda (v-1) \equiv 0 \!\!\!\!\pmod{5},$$
which implies the thesis.
\end{proof}

\section{Costructions and related structures}

In this section we will introduce some useful definitions. For missing
terms or results that are not explicitly explained in the paper,
the reader is referred to \cite{CD} and its online updates.
For some results below, we also cite this handbook instead of the
 original papers.

 \bigskip

\bigskip

A (resolvable) $G$-decomposition of the complete multipartite
graph with $u$ parts each of size $g$  is known as a (resolvable)
group divisible design $G$-(R)GDD of type $g^u$ (the parts of size
$g$ are called the \textit{groups} of the design). When $G = K_n$ we will
call it an $n$-(R)GDD.
If the blocks of a $G$-GDD of type ${g^u}$ can be partitioned into partial
parallel classes, each of which contains all points except those of one group,
we refer to the decomposition as a {\em frame}. It is easy to deduce that the number of partial parallel classes missing a specified group  is
$\frac{g |V(G)|}{2 |E(G)|}$.

\bigskip

An incomplete resolvable $G$-design of order $v+h$ and index $\lambda$ with  a hole of size $h$ denoted
by $G$-IRD$(v+h,h,\lambda)$, is a
$G$-decomposition of $\lambda (K_{v+h}\setminus K_h)$ in which there are two types of
classes, $\frac{\lambda (h-1) |V(G)|}{2 |E(G)|}$ {\em partial} classes which cover every point except
those in the hole (the set of points of $K_h$ are referred to as the {\em hole}) and  $\frac{\lambda v|V(G)| }{2 |E(G)|}$ {\em full} classes which cover every point of $K_{v+h}$.

\section{Small cases}

\begin{lemma}
\label{lemmaA1} There exists a resolvable $(2 K_{4}, C_{4})$-design.
\end{lemma}
\begin{proof}
Let $V$= $\{0,1,2,3\}$ be  the vertex set and consider the classes listed below:\\
$\{(0,1,2,3)\}$, $\{(0,1,3,2)\}$, $\{(0,2,1,3)\}$.
\end{proof}

\begin{lemma}
\label{lemmaA2} There exists a resolvable $(2 K_{4}, K_{3}+e)$-design.
\end{lemma}
\begin{proof}
Let  $V$= $\{0,1,2,3\}$  be  the vertex set and consider  the classes listed below:\\
$\{(0,2,3-1)\}$, $\{(3,2,1-0)\}$, $\{(2,1,0-3)\}$.
\end{proof}

\begin{lemma}
\label{lemmaA2bis} There exists a resolvable $(2 K_{8}, K_{3}+e)$-design.
\end{lemma}
\begin{proof}
Let $V$= $Z_{7}\cup \{\infty \}$  be  the vertex set. The desired design is obtained by developing in $Z_{7}$ the 
following base blocks: 
$\{(\infty, 1,5-6),(0,4,2-3)\}$.
\end{proof}

\begin{lemma}
\label{lemmaA3} There exists a resolvable $(2 K_{4},K_{1,3})$-design.
\end{lemma}
\begin{proof}
Let $V$= $Z_{4}$ be the vertex set. The desired design is obtained by developing in $Z_{4}$ the  base block 
$\{(0;1,2,3)\}$.
\end{proof}

\begin{lemma}
\label{lemmaA4} There exists a resolvable $K_{1,3}$-RGDD of type $4^{2}$ and index $6$.
\end{lemma}
\begin{proof}
Take  $\{x_1,x_2,x_3, x_4\}$ and $\{y_1,y_2,y_3, y_4\}$ as groups and consider the classes obtained by developing the following base blocks, reducing subscripts modulo 4:\\
$\{(x_1;y_1, y_2,y_3), (y_4; x_2, x_3, x_4)\}$, $\{(x_1;y_2, y_3,y_4), (y_1; x_2, x_3, x_4)\}$, $\{(x_1;y_3, y_4,y_1), (y_2; x_2, $ $x_3, x_4)\}$, $\{(x_1;y_4, y_1, y_2), (y_3; x_2, x_3, x_4)\}$.
\end{proof}

\begin{lemma}
\label{lemmaA5} There exists a resolvable $K_{1,3}$-RGDD of type $4^{3}$ and index $3$.
\end{lemma}
\begin{proof}
Take $3Z_{12} +i$, $i=0,1,2$, as groups and the blocks obtained by developing $(2; 0,7,9)$,  which  gives four partial parallel classes, and the base blocks 
$(4; 0,3,8)$,  $(6; 1,2,5)$, $(9; 7,10,11)$ (giving a full parallel classe).
\end{proof}

\begin{lemma}
\label{lemmaA6} There exists a resolvable $(6K_{8},K_{1,3})$-design.
\end{lemma}
\begin{proof}
Start with a resolvable $K_{1,3}$-RGDD of type $4^{2}$ and index $6$, which exists by
Lemma \ref{lemmaA4}, and fill each group of size 4 with a copy of a resolvable $(6K_{4},K_{1,3})$-design, which exists by
Lemma \ref{lemmaA3}. 
\end{proof}

\begin{lemma}
\label{lemmaA61} There exists a resolvable $(6K_{12},K_{1,3})$-design.
\end{lemma}
\begin{proof}
Start with a resolvable $K_{1,3}$-RGDD of type $4^{3}$ and index $6$, which exists by
Lemma \ref{lemmaA5}, and fill each group of size 4 with a copy of a resolvable $(6K_{4},K_{1,3})$-design, which exists by
Lemma \ref{lemmaA3}. 
\end{proof}

\begin{lemma}
\label{lemma20} There exists a resolvable $(6K_{20},K_{1,3})$-design.
\end{lemma}
\begin{proof}
Let $V$= $Z_{19}\cup \{\infty \}$ be the vertex set. The desired design is obtained by developing in $Z_{19}$ the following base blocks:\\
$\{(\infty; 4,5,12), (0;3,9,11), (8;1,2,13), (10; 7,14,18), (15; 6,16,17)\}$,\\
$\{(11;\infty, 9,12), (0;3,4,5), (8;1,2,13), (10; 7,14,18), (15; 6,16,17)\}$, \\
$\{(5;\infty, 3,11), (0;4,9,12), (8;1,2,13), (10; 7,14,18), (15; 6,16,17)\}$,\\
$\{(3;\infty, 9,4), (0;5,12,11), (8;1,2,13), (10; 7,14,18), (15; 6,16,17)\}$.
\end{proof}

\begin{lemma}
\label{lemmaA7} There exists a resolvable $(6K_{24},K_{1,3})$-design.
\end{lemma}
\begin{proof}
Start with a resolvable 2-RGDD of type $1^{6}$. Give weight 4 to all points and  replace each edge of a given resolution class with a copy of a  resolvable $K_{1,3}$-RGDD of type $4^{2}$ and index $6$, which exists by Lemma \ref{lemmaA4}. Finally, fill each group of size 4 with a copy of a resolvable $(6K_{4},K_{1,3})$-design, which exists by Lemma \ref{lemmaA3}. 
\end{proof}

\begin{lemma}
\label{lemmaA8} There exists a resolvable $(6K_{36},K_{1,3})$-design.
\end{lemma}
\begin{proof}
Start with a resolvable resolvable $S(2,3,9)$. Give weight 4 to all points and  replace each block of a given resolution class with  a copy of a resolvable $K_{1,3}$-RGDD of type $4^{3}$ and index $6$, which exists by Lemma \ref{lemmaA5}. Finally, fill each group of size 4 with a copy of a resolvable $(6K_{4},K_{1,3})$-design, which exists by Lemma \ref{lemmaA3}. 
\end{proof}

\begin{lemma}
\label{lemmaA10} There exists a resolvable $(5 K_{4}, K_{4}-e)$-design.
\end{lemma}
\begin{proof}
Let $V$= $\{0,1,2,3\}$ be the vertex set and consider the classes listed below:\\
$\{(1,2,0;3)\}$, $\{(3,0,2;1)\}$, $\{(1,3,0;2)\}$, $\{(2,0,1;3)\}$, $\{(1,0,2;3)\}$, $\{(2,3,1;0)\}$.
\end{proof}

\begin{lemma}
\label{lemmaA11} There exists a resolvable $(5K_{8},K_{4}-e)$-design.
\end{lemma}
\begin{proof}
Let $V$= $Z_{7}\cup \{\infty \}$ be the vertex set. The desired design is obtained by developing in $Z_{7}$ the following base blocks:\\
$\{(\infty, 4,5;0), (1,2,3;6)\}$,  $\{(0,3,\infty; 2), (1,6,5;4)\}$.
\end{proof}

\begin{lemma}
\label{lemmaA12} There exists a resolvable $(5K_{12},K_{4}-e)$-design.
\end{lemma}
\begin{proof}
Let $Z_{8}\cup \{\infty_1,\infty_2,\infty_3,\infty_4\}$ be the vertex set. The desired design is obtained by filling the sets $2Z_{8} +i$, $i=0,1$, and $ \{\infty_1,\infty_2,\infty_3,\infty_4\}$  with a copy of a resolvable $(5 K_{4}, K_{4}-e)$-design (giving  six parallel classes) and developing the  two sets of base blocks in $Z_{8}$ 
$\{(\infty_1,1,0;2)$, $(\infty_2,4,3;5)$, $(6,7,\infty_3;\infty_4)\}$ and 
$\{(\infty_3,3,0;6)$, $(\infty_4,4,1;7)$, $(2,5,\infty_1;\infty_2)\}$,
which give the remaining parallel classes.
\end{proof}

\begin{lemma}
\label{lemmaA15} There exists a resolvable $(5K_{20},K_{4}-e)$-design.
\end{lemma}
\begin{proof}
Let $V$= $Z_{19}\cup \{\infty \}$ be the vertex set. The desired design is obtained by developing in $Z_{19}$ the following base blocks:\\
$\{(3,4,15;\infty), (1,18,9;14),(2,0,5;8),(6,10,12;13),(7,16,11;17)\}$,\\
$\{(0,\infty, 1;15), (8,10,18;5),(2,16,11;9),(14,13,17;7),(4,6,12;3)\}$.
\end{proof}

\begin{lemma}
\label{lemmaA13} There exists a resolvable $(5K_{24},K_{4}-e)$-design.
\end{lemma}
\begin{proof}
Start with a resolvable $(K_4-e)$-RGDD of type $4^{6}$ of index 5 (see \cite{GL}) and fill each group of size 4 with a copy of a resolvable $(5K_{4},K_{4}-e)$-design, which exists by  Lemma \ref{lemmaA10}. \end{proof}

\begin{lemma}
\label{lemma28-8} There exists an incomplete resolvable $(K_4-e)$-design of order $28$ and index $\lambda=5$  with a hole of size $8$.
\end{lemma}
\begin{proof} 
Let the vertex set be $Z_{20}\cup \{\infty_1,\infty_2,\ldots,\infty_8\}$. The desired design is obtained by filling each set $5Z_{20} +i$, $i=0,1,2,3,4$, with a copy of a resolvable $(5 K_{4}, K_{4}-e)$-design and developing the base blocks $(0, 3,1;2)$, $(0, 7,6;9)$, each of which gives  four partial parallel classes, and the following base blocks (partitioned into  full parallel classes): \\
$\{(2, 3,\infty_1;\infty_2)$, $(5, 11,\infty_3;\infty_4)$, $(\infty_5,9, 10;12)$, 
$(\infty_6,1, 18;19)$, $(\infty_7,6, 15;17)$, $(\infty_8,$ $7, 13;14)$, $(0,8, 4;16)\}$;\\
$\{(10, 3,\infty_5;\infty_6)$, $(5, 11,\infty_7;\infty_8)$, $(\infty_1,9, 2;12)$, 
$(\infty_2,1, 18;19)$, $(\infty_3,6, 15;17)$, $(\infty_4,$ $7, 13;14)$, $(0,16, 4;8)\}$. 
\end{proof}

\begin{lemma}
\label{lemmaA14} There exists a resolvable $(5K_{36},K_{4}-e)$-design.
\end{lemma}
\begin{proof}
Take a resolvable $(K_{36},K_{4}-e)$-design (see \cite{GL}) and repeat the classes 5 times.
\end{proof}

\begin{lemma}
\label{lemma44} There exists a resolvable $(5 K_{44}, K_{4}-e)$-design.
\end{lemma}
\begin{proof} It is sufficient to paste five copies of the maximun resolvable  ($K_4-e$)-packing of order $44$ in \cite{WS},  with  leaves such that their
 edges can be suitably arreanged so to give  a new parallel class of copies of $K_{4}-e$.
\end{proof}

\begin{lemma}
\label{lemma68} There exists a resolvable $(5 K_{68}, K_{4}-e)$-design.
\end{lemma}
\begin{proof} Consider  the minimum resolvable  ($K_4-e$)-covering of order $68$ in \cite{SW},  whose excess is contained into one parallel class so that removing its edges gives a maximun resolvable  ($K_4-e$)-packing of order $68$ with  leave a parallel class of kites. Now, paste five copies of the above packing with leaves such that their
 edges can be suitably arreanged into four  new parallel classes of copies of $K_{4}-e$.
\end{proof}

\section{The case $G=C_{4}, K_{3}+e$}

\begin{lemma}
\label{lemmaC1} For every \ $v\equiv 0\pmod{4}$ there exists
a resolvable $( 2 K_{v},C_{4})$-design.
\end{lemma}

\begin{proof}

Let $v=4t$. Start from a 2-RGDD $\cG$ of type $2^{t}$ (\cite{CD}). Give weight $2$ to
every point of $\cG$ and for each block of a given
resolution class of $\cG$ place 2 copies  of a $C_{4}$-decomposition  of type
$2^{2}$ (\cite{DQS}). Finally, fill each group of size 4 with  a copy of a resolvable $(2 K_{4},C_{4})$-design
 which exists by Lemma  \ref{lemmaA1}.
\end{proof}

\begin{lemma}
\label{lemmaC2} For every \ $v\equiv 0\pmod{4}$  there exists
a resolvable $(2 K_{v}, K_{3}+e)$-design.
\end{lemma}

\begin{proof}

Let $v=4t$. The cases $t=1$ and $t=2$ correspond to a  $(2K_{4}, K_{3}+e)$-design and a  $(2K_{8}, K_{3}+e)$-design  which exist by Lemma  \ref{lemmaA2} and Lemma \ref{lemmaA2bis}.\\
For $t>2$, start from a $(K_{3}+e)$-RGDD $\cG$ of type $4^{t}$ (\cite{W}). Repeat each block of a given
resolution class of $\cG$  twice and fill each group of size 4 with a copy of a a resolvable $(2 K_{4},K_{3}+e)$-design which exists by Lemma  \ref{lemmaA2}. 
\end{proof}

\section{The case $G= K_{1,3}$}

\begin{lemma}
\label{lemmaD1} For every \ $v\equiv 4\pmod{12}$, there exists
a resolvable $(2 K_{v},K_{1,3})$-design.
\end{lemma}
\begin{proof}

Start from a resolvable $(K_v, K_4)$-design   (\cite{HRW}) and replace each block of a given resolution class  with
2  copies of a resolvable $(2 K_{4},K_{1,3})$-design which exists by Lemma  \ref{lemmaA3}. This completes the proof.
\end{proof}

\begin{lemma}
\label{lemmaD2} For every \ $v\equiv 0,8\pmod{12}$ there exists
a resolvable $(6 K_{v}, K_{1,3})$-design.
\end{lemma}

\begin{proof}
The cases $v=8,12,20,24,36$ are covered by Lemmas \ref{lemmaA6}, \ref{lemmaA61}, \ref{lemma20}, \ref{lemmaA7}, \ref{lemmaA8}.  We distingish the following cases.

\begin{itemize}
\item[Case $1$]: $v\equiv 0\pmod{12}$, $v\geq 48$.

Let $v=12t$.  Start from a $4$-RGDD $ \cD$ of type $12^{t}$  (\cite{CD}) and replace each block of a given resolution class of $\cD$ with 3 copies of a resolvable $(2 K_{4},K_{1,3})$-design which exists by Lemma  \ref{lemmaA3}. Finally, fill each group of size 12 with a copy  of a  resolvable $(6 K_{12},K_{1,3})$-design which exists by Lemma  \ref{lemmaA61}. 
\item[Case $2$]: $v\equiv 8\pmod{24}$, $v\geq 32$.

Let $v=8+24t$. Start from a $4$-RGDD $ \cD$ of type $8^{1+3t}$  (\cite{CD}) and replace each block of a given resolution class of $\cD$ with 3 copies of a resolvable $(2 K_{4},K_{1,3})$-design which exists by Lemma  \ref{lemmaA3}. Finally, fill each group of size 12 with a copy of a  resolvable $(6 K_{8},K_{1,3})$-design which exists by Lemma  \ref{lemmaA6}. 

\item[Case $3$:]  $v\equiv 20\pmod{24}$, $v\geq 44$.

Let $v=20+24t$.  Start from a $2$-frame  $\cF$ of type $2^{2+3t}$,$t\geq 1$,  with groups $G_i$, $i=1,2,\ldots,2+3t$, (\cite{CD}); let $X=\cup_{i=1}^{2+3t} G_i$. Expand each point of $X$ $4$ times and add a set $H=\{\infty_1, \infty_2, \infty_3, \infty_4\}$. For each $x\in X$, place on $\{x\}\times Z_4$ 3 copies of a resolvable $(2 K_{4},K_{1,3})$-design which exists by Lemma  \ref{lemmaA3}. This gives a set $P$ of 12 partial parallel classes on $X \times Z_4$. Fill the hole $H$ with 3  copies of a resolvable $(2 K_{4},K_{1,3})$-design which exists by Lemma  \ref{lemmaA3}. This gives a set $P_H$ of 12 parallel classes on $H$. Combine $P$ and $P_H$  to obtain 12 full parallel classes of 3-stars. For $j=1,2$,  let $p_{i,j}$ be the 2 partial parallel classes which miss the group $G_i$ and for each $b\in p_{i,j}$, place on $b\times Z_4$ a copy    of a resolvable $K_{1,3}$-RGDD of type $4^{2}$ and index $6$, which exists by Lemma  \ref{lemmaA4}. This gives a set $P_{i,j}$ of $16$  partial classes of 3-stars on $X\setminus G_i$. For each $i=1,2,\ldots,2+3t$ place on $H\cup (G_i\times Z_4)$ a copy  of a resolvable $K_{1,3}$-RGDD of type $4^{3}$ and index $6$, which exists by Lemma \ref{lemmaA5}; this gives a set $P_i$ of $32$  classes of 3-stars. Finally, combine the $32$  classes of $P_i$ with the classes of $P_{i,1}$ and $P_{i,2}$ to obtain the desired result. 
\end{itemize} \end{proof}

\section{The case $G=K_4-e$}

\begin{lemma}
\label{lemmaE1} For every \ $v\equiv 0,4,8,12,16\pmod{24}$, there exists
a resolvable $(5 K_{v}, $ $K_{4}-e)$-design.
\end{lemma}

\begin{proof}
The cases $v=4,8,12,24,36$ are covered by Lemmas \ref{lemmaA10}, \ref{lemmaA11}, \ref{lemmaA12}, \ref{lemmaA13}, \ref{lemmaA14} and \ref{lemmaA15}.  We distinguish the following cases.

\begin{itemize}
\item
[Case $1$]: $v\equiv 4 \pmod{12}$.

Start from a resolvable $(K_v, K_4)$-design (\cite{HRW}) and replace each block of every resolution class  with
a copy of a resolvable $(5K_{4},K_{4}-e)$-design which exists by Lemma  \ref{lemmaA10}. 

\item
[Case $2$]: $v\equiv 0\pmod{12}$, $v\geq 48$.

Let $v=12t$. Start from a $4$-RGDD $ \cD$ of type $12^{t}$  (\cite{CD}). Replace each block of every resolution class of $\cD$ with
a copy of a resolvable $(5K_{4},K_{4}-e)$-design which exists by Lemma  \ref{lemmaA10} and fill each group of size 12 with
a copy of a  resolvable $(5 K_{12},K_{4}-e)$-design which exists by Lemma  \ref{lemmaA12}. 
  
\item
[Case $3$]: $v\equiv 8\pmod{24}$, $v\geq 32$.

Let $v=8+24t$.  Start from a $4$-RGDD $ \cD$ of type $8^{1+3t}$  (\cite{CD}). Replace each block of every resolution class of $\cD$ with
 a copy of a resolvable $(5K_{4},K_{4}-e)$-design which exists by Lemma  \ref{lemmaA10} and fill each group of size 8 with a  copy of a  resolvable $(5 K_{8},K_{4}-e)$-design which exists by Lemma  \ref{lemmaA11}. This completes the proof.
\end{itemize} \end{proof}

\begin{lemma}
\label{lemmaE3} For every \ $v\equiv 20\pmod{120}$, there exists
a resolvable $(5K_{v}, K_{4}-e)$-design.
\end{lemma}

\begin{proof}
Let $v=20+120t$. The case $t=0$ corresponds to a  $(5K_{20}, K_{4}-e)$-design which exists by Lemma  \ref{lemmaA15}. For $t>1$ take  a $4$-RGDD $ \cG$ of type $4^{1+6t}$  (\cite{CD}).  Expand each point  $5$ times and replace each block of the resolution classes of $\cG$ with
 a copy of a resolvable $(K_{4}-e)$-RGDD of tipe $5^{4}$ and index 5  which exists by Lemma 3.2 of \cite{WS}. Finally, fill each group of size 20 with a  copy of a resolvable $(5 K_{20},K_{4}-e)$-design which exists by Lemma  \ref{lemmaA15}. 
\end{proof}

\begin{lemma}
\label{lemmaE4} For every \ $v\equiv 44\pmod{120}$, $v>44$, there exists
a resolvable $(5K_{v}, K_{4}-e)$-design.
\end{lemma}

\begin{proof}
Let $v=44+120t$. The case $t=0$ corresponds to a  $(5K_{44}, K_{4}-e)$-design which exists by Lemma  \ref{lemma44}. For $t>0$ take  a $(K_4-e)$-RGDD $ \cG$ of type $4^{1+5(2+6t)}$ and index 5 (\cite{WS}) and fill each group of size 4 with a  copy of a resolvable $(5 K_{4},K_{4}-e)$-design which exists by Lemma  \ref{lemmaA10}. 
\end{proof}

\begin{lemma}
\label{lemmaE6} For every \ $v\equiv 68\pmod{120}$, there exists
a resolvable $(5K_{v}, K_{4}-e)$-design.
\end{lemma}
\begin{proof}
Let $v$=$68+120t$. The case $t=0$ corresponds to a  $(5K_{68}, K_{4}-e)$-design which exists by Lemma  \ref{lemma68}. For $t>0$,
let $\cF$ be a $(K_4-e)$-frame of type $20^{3+6t}$ (\cite{WS}) with groups $G_i$, $i=1,2,\ldots,3+6t$.  Add a set $H=\{\infty_1,\infty_2,\ldots,\infty_8\}$.
For each  $i=1,2,\ldots,3+6t$,  let $P_{i}$ the set of the  partial parallel classes which miss the group $G_i$, taken 5 times.  Place on $G_i\cup H$ a copy of an incomplete resolvable $(K_4-e)$-design of order 28 and index 5, having the set of 8 infinite poits as hole, and  combine its full classes with the partial classes of $P_i$ so to obtain $40(3+6t)$ parallel classes on $H\cup(\cup_{i=1}^{3+6t} G_i)$. Fill the hole $H$ with a copy  of  a resolvable $(5K_{8}, K_{4}-e)$-design which exists by Lemma  \ref{lemmaA11} and combine its 14 classes with the partial classes of $P_i$  so to obtain 14 parallel classes. The result is a resolvable $(5K_{v}, K_{4}-e)$-design.
\end{proof}

\begin{lemma}
\label{lemmaE5} For every \ $v\equiv 92\pmod{120}$, there exists
a resolvable $(5K_{v}, K_{4}-e)$-design.
\end{lemma}

\begin{proof}
Let $v=92+120t$. Let $P_1=\{\{i,1+i\}, i \in Z_v\}$ and $P_{\frac{v}{2}-1}=\{\{i,\frac{v}{2}-1+i\}, i \in Z_v\}$ be two sets of $v$ edges of $K_v$ and
$P_{\frac{v}{2}}=\{\{i,\frac{v}{2}+i\}, i=0,1,\ldots,\frac{v}{2}-1\}$ be a set of $\frac{v}{2}$ edges of $K_v$. The edges of $P_1\cup P_{\frac{v}{2}-1}\cup P_{\frac{v}{2}}$ can be decomposed into a set of 5 1-factors $F_j, j=1,2,3,4,5$ (\cite{SL}). For each $j=1,2,3,4,5$
take on $K_v$  a $(K_4-e)$-RGDD $ \cG$ of type $2^{6+10(4+6t}$  (\cite{WS}) with the edges of $F_j$ as groups.  Finally, the missing edges of $P_1\cup P_{\frac{v}{2}-1}\cup P_{\frac{v}{2}}$ (until now arranged only 4 times) can be decomposed into the set $\{(i,\frac{v}{2}+i, \frac{v}{2}+1+i; 1+i)$, $i =0,1,\ldots, \frac{v}{2}-1$\} of $\frac{v}{2}$ copies of $K_4-e$,  which generate the remaining two parallel classes (the first class is obtained for $i =0,2,\ldots, \frac{v}{2}-2$, while the second class for $i =1,3,\ldots, \frac{v}{2}-1$).
\end{proof}

\begin{lemma}
\label{lemmaE2} For every \ $v\equiv 116\pmod{120}$, there exists
a resolvable $(K_{v}, K_{4}-e)$-design.
\end{lemma}

\begin{proof}
Since the condition $v\equiv 116\pmod{120}$ implies $v\equiv 16\pmod{20}$, the proof follows by \cite{GL,W1}.
\end{proof}

\section{Conclusion}

We are now in a position to prove the following main result.

\begin{thm}
The necessary conditions of Lemma \ref{lemmaP} are sufficient for the existence of a resolvable $(\lambda K_{v},G)$-design in the cases when $G= C_4, K_3+e, K_{1,3}, K_4-e$.
\end{thm}
\begin{proof}
Let $G\in \{C_{4}, K_{3}+e\}$. For 
 any  \ $v \equiv 0\pmod {4}$ and $\lambda =2\mu$ ($\mu \geq 1$), combine $\mu$ copies of a  resolvable $(2 K_{v},G)$-design, which exists by Lemmas  \ref{lemmaC1} and \ref{lemmaC2}.
 
 Let $G=K_{1,3}$. For 
 any  \ $v \equiv 4\pmod {12}$ and $\lambda =2\mu$ ($\mu \geq 1$), combine $\mu$ copies of a  resolvable $(2 K_{v},K_{1,3})$-design, which exists by Lemma  \ref{lemmaD1}; while,   
 for any  \ $v \equiv 0,8\pmod {12}$ and $\lambda =6\mu$ ($\mu \geq 1$), combine $\mu$ copies of a  resolvable $(6 K_{v},K_{1,3})$-design, which exists by Lemma  \ref{lemmaD2}.
 
 Let $G=K_4-e$. For 
 any  \ $v \equiv 0,1,8,12 \pmod {20}$ and $\lambda =5\mu$ ($\mu \geq 1$), combine $\mu$ copies of a  resolvable $(5 K_{v},K_4-e)$-design, which exists by Lemmas  \ref{lemmaE1} -\ref{lemmaE6}; while,   
 for any  \ $v \equiv 16\pmod {20}$ combine $\lambda$ copies of a  resolvable $(K_{v},K_4-e)$-design, which exists by Theorem 3 of  \cite{W1}.
\end{proof}

\end{document}